\documentclass[11pt]{article}

\usepackage{amssymb,amsmath,amsfonts}
\usepackage{graphicx,color,enumitem}
\usepackage{amsthm} 
\usepackage{bm}
\usepackage[round]{natbib}

\usepackage{geometry}

\usepackage[colorinlistoftodos, textwidth=4cm, shadow]{todonotes}
\RequirePackage[colorlinks,citecolor=blue,urlcolor=blue]{hyperref}

\usepackage{tikz}
\tikzstyle{vertex}=[circle, draw, inner sep=2pt, fill=white]

\newcommand{\E}{{\mathbb E}}
\newcommand{\F}{{\mathbb F}}

\renewcommand{\P}{{\mathbb P}}
\newcommand{\Q}{{\mathbb Q}}

\newcommand{\R}{{\mathbb R}}

\newcommand{\N}{{\mathbb N}}

\newcommand{\Ecal}{{\mathcal E}}
\newcommand{\Fcal}{{\mathcal F}}
\newcommand{\Gcal}{{\mathcal G}}

\newcommand{\Lcal}{{\mathcal L}}

\newcommand{\Scal}{{\mathcal S}}

\newcommand{\Xcal}{{\mathcal X}}

\newcommand{\Mid}{{\ \Big|\ }}

\newcommand{\fdot}{{\,\cdot\,}}

\DeclareMathOperator{\conv}{conv}
\DeclareMathOperator{\cconv}{\overline{conv}}
\DeclareMathOperator{\aff}{aff}
\DeclareMathOperator{\caff}{\overline{aff}}

\newtheorem{theorem}{Theorem}

\newtheorem{corollary}[theorem]{Corollary}
\newtheorem{definition}[theorem]{Definition}
\newtheorem{example}[theorem]{Example}
\newtheorem{lemma}[theorem]{Lemma}
\newtheorem{proposition}[theorem]{Proposition}
\newtheorem{remark}[theorem]{Remark}

\numberwithin{equation}{section}
\numberwithin{theorem}{section}

\begin{document}

\title{Conditional infimum and recovery of monotone processes}
\author{Martin Larsson\thanks{Department of Mathematics, ETH Zurich, R\"amistrasse 101, CH-8092, Zurich, Switzerland, martin.larsson@math.ethz.ch.}\ \thanks{ The author would like to thank Nicole El Karoui, Pietro Siorpaes, and Josef Teichmann for useful comments and fruitful discussions. Financial support by the Swiss National Science Foundation (SNF) under grant 205121\textunderscore163425 is gratefully acknowledged.}}

\maketitle

\begin{abstract}
Monotone processes, just like martingales, can often be recovered from their final values. Examples include running maxima of supermartingales, as well as running maxima, local times, and various integral functionals of sticky processes such as fractional Brownian motion. An interesting corollary is that any positive local martingale can be reconstructed from its final value and its global maximum. These results rely on the notion of conditional infimum, which is developed for a large class of complete lattices. The framework is sufficiently general to handle also more exotic examples, such as the process of convex hulls of certain multidimensional processes, and the process of sites visited by a random walk.
\\[2ex] 
\\[2ex]
\end{abstract}

\section{Introduction}

Let $M=(M_t)_{t=0,\ldots,T}$ be a discrete time martingale defined on a filtered probability space $(\Omega,\Fcal,\{\Fcal_t\}_{t=0,\ldots,T},\P)$, where we suppose $\Omega$ is finite. Let $\overline M_t=\max_{s\le t}M_s$ be the running maximum. Just as the martingale $M$ can be recovered from its final value $M_T$ via the formula $M_t=\E[M_T\mid\Fcal_t]$, the running maximum process $\overline M$ can be recovered from its final value $\overline M_T$. In fact, for any $t\in\{0,\ldots,T\}$ and non-null $\omega\in\Omega$, we claim that
\begin{equation} \label{intro1}
\overline M_t(\omega)=\min_{\omega'\in A}\overline M_T(\omega'),
\end{equation}
where $A$ is the atom of $\Fcal_t$ containing $\omega$. To see this, note that $\overline M_t(\omega)\le \min_{\omega'\in A}\overline M_T(\omega')$ since $\overline M$ is nondecreasing. If the inequality were strict, we would have $M_t\le \overline M_t<\overline M_T$ on the $\Fcal_t$-measurable event $A$, contradicting the martingale property: on $A$, $M$ would be sure to experience a strict increase between $t$ and $T$. Thus \eqref{intro1} must hold.

The right-hand side of the identity \eqref{intro1} is the {\em conditional infimum of $\overline M_T$ given $\Fcal_t$}, evaluated at~$\omega$, and the identity itself expresses an ``inf-martingale'' property of $\overline M$. The goal of the present paper is to develop these ideas in some generality. For a large class of complete lattices $S$, we show that the conditional infimum of an $S$-valued random element $X$ given a sub-$\sigma$-algebra $\Ecal$ is well-defined; we denote it by $\bigwedge[X\mid\Ecal]$. In the presence of a filtration one is led to consider ``inf-martingales'' $\bigwedge[X\mid\Fcal_t]$, $t\ge0$, and a key message of this paper is that many naturally occurring nondecreasing processes turn out to have this property. They can then be recovered from their final value. Examples include running maxima of supermartingales and, more generally, of processes that become supermartingales after an equivalent change of measure (Proposition~\ref{P:maxM}). Running maxima, local times, and various integral functionals of so-called sticky processes also have this property (Propositions~\ref{P:f(X) run max}, \ref{P:X U K sticky}, \ref{P:int f sticky}, and their corollaries). More exotic examples include the process of convex hulls of certain multidimensional processes, and the process of sites visited by a random walk (Propositions~\ref{P:convex hull sticky} and~\ref{P:RV sticky}). These results are derived from a simple ``no-arbitrage'' principle for nondecreasing lattice-valued processes (Theorem~\ref{T:NCR}). In the martingale context, an interesting corollary is that any positive local martingale can be recovered from its final value and its global maximum (Proposition~\ref{P:maxM X}).

The general theory covers a rather broad class of measurable complete lattices $S$. One only needs measurability of the ``triangle'' $\{(x,y)\colon x\le y\}$ in the product space $S\times S$, measurability of the countable supremum and infimum maps, and existence of a strictly increasing measurable map $S\to\R$. These hypotheses are stated precisely in \ref{A1}--\ref{A3} below. Apart from the extended real line $[-\infty,\infty]$, we prove that this covers the family of closed convex subsets of $\R^d$, as well as the family $2^{\Xcal}$ of subsets of a countable set $\Xcal$ (Theorems~\ref{T:convex sets} and~\ref{T:countable set}, respectively).

Conditional infima (and suprema) for real-valued random variables have appeared previously in the literature, along with real-valued ``inf-martingales'' (or ``sup-martingales'', also called max-martingales or maxingales); see for instance \cite{bar_car_jen_03,elk_mez_08}. We extend these constructions to general complete lattices with the additional structural properties mentioned above. A related but different notion of maxingale has been used by \cite{puh_97,puh_99,puh_01} and \cite{fle_04} in the context of idempotent probability with applications to large deviations theory and control theory. The notion of stickiness, which is closely related to the developments in the present paper, plays an important role in mathematical finance; see e.g.~\citet{GRS:08,BPS:15,RS:16}. Conditional infima in lattices of sets have also been useful in problems from multidimensional martingale optimal transport; see \cite{OS:17}, who make use of our Example~\ref{E:CI 2} below.

The rest of the paper is organized as follows. After ending this introduction with some remarks on notation, we turn to Section~\ref{S:CI} where the general theory of conditional infima in complete lattices is developed, including analogues of the martingale regularization and optional stopping theorems. Section~\ref{S:sticky} discusses sticky processes and their relations to conditional infima. Applications to martingale theory are given in Section~\ref{S:martingales}, including a general version of \eqref{intro1}. Examples involving processes of convex hulls and processes of subsets of a countable set a given in Section~\ref{S:further}. Section~\ref{S:closed subsets} develops several general results, mainly of measure theoretic nature, for lattices of closed sets. These results should be of independent interest.

\subsection{Remarks on notation}

Throughout this paper, $(\Omega,\Fcal,\P)$ is a probability space. Relations between random quantities are understood in the almost sure sense, unless stated otherwise. The probability space is endowed with a filtration $\F=\{\Fcal_t\}_{t\ge 0}$ of sub-$\sigma$-algebras of $\Fcal$, and we set $\Fcal_\infty=\bigvee_{t\ge0}\Fcal_t$. The filtration $\F$ need not be augmented with the $\P$-nullsets, but unless stated otherwise it is assumed throughout the paper that $\F$ is right-continuous. It is sometimes convenient to work with the order-theoretic indicator $\chi_A$ of a subset $A\subseteq\Omega$, defined by
\[
\chi_A(\omega) = \begin{cases} -\infty, & \omega\in A\\ +\infty, & \omega\notin A. \end{cases}
\]
The meaning of the symbols $+\infty$ and $-\infty$ are discussed below.

\section{Conditional infimum} \label{S:CI}

Throughout this section, let $(S,\le)$ be a complete lattice. That is, $S$ is a partially ordered set such that any subset $A\subseteq S$ has a least upper bound, denoted by $\sup A$. This implies that the greatest lower bound $\inf A$ also exists, and that $S$ contains a greatest element $+\infty$ and smallest element $-\infty$. We write $x\vee y$ for $\sup\{x,y\}$ and $x\wedge y$ for $\inf\{x,y\}$, and use $x<y$ as shorthand for $x\le y$ and $x\ne y$.

We assume that $S$ is equipped with a $\sigma$-algebra $\Scal$ that satisfies the following two properties:
\begin{enumerate}[label={\rm(A\arabic*)}]
\item\label{A1} The set $\{(x,y)\in S^2\colon x\le y\}$ lies in the product $\sigma$-algebra $\Scal^2=\Scal\otimes\Scal$.
\item\label{A2} The countable supremum and infimum maps
\[
(x_1,x_2,\ldots) \mapsto \sup\{x_1,x_2,\ldots\} \qquad\text{and}\qquad
(x_1,x_2,\ldots) \mapsto \inf\{x_1,x_2,\ldots\}
\]
are measurable, where the set of sequences $S^\infty=\{(x_1,x_2,\ldots)\colon x_n\in S \text{ for all $n$}\}$ is equipped with the product $\sigma$-algebra $\Scal^\infty=\bigotimes_{n=1}^\infty\Scal$.
\end{enumerate}
These properties ensure that random elements of~$S$ (i.e., measurable maps $\Omega\to S$) behave well. Indeed, let $X_n$, $n\in\N$, be random elements of $S$. Assumption~\ref{A1} implies that $\{X_1\le X_2\}\in\Fcal$, and hence also $\{X_1<X_2\}\in\Fcal$.\footnote{Indeed, $\{X_1<X_2\}$ equals $\{X_1\le X_2\}\setminus(\{X_1\le X_2\}\cap\{X_2\le X_1\})$ and is therefore measurable.} Assumption~\ref{A2} implies that $\sup_n X_n$ and $\inf_n X_n$ are again random elements of $S$. This will be used repeatedly in what follows.

Finally, we make the following assumption, where, of course, strictly increasing means that $x<y$ implies $\phi(x)<\phi(y)$:
\begin{enumerate}[label={\rm(A\arabic*)},resume]
\item\label{A3} There exists a strictly increasing measurable map $\phi\colon S\to\R$.
\end{enumerate}

\begin{remark}
In some cases, naturally appearing lattices are not complete, but only {\em Dedekind complete}: suprema (infima) are guaranteed to exist only for subsets $A\subseteq S$ that are bounded above (below). In such cases one can extend the given lattice to a complete lattice satisfying \ref{A1}--\ref{A3}, provided these properties hold in the given lattice; see Proposition~\ref{P:ext Dedekind}.
\end{remark}

There are plenty of examples of complete lattices which satisfy \ref{A1}--\ref{A3}, some of which are discussed below.  The first example below concerns the familiar (extended) real-valued case, while the subsequent examples involve more complicated complete lattices.

\begin{example} \label{E:CI 1}
$\overline\R=[-\infty,\infty]$ together with the usual order and the Borel $\sigma$-algebra is a complete lattice which clearly satisfies \ref{A1}--\ref{A3}.
\end{example}

\begin{example} \label{E:CI 2_new}
Let $\Xcal$ be a countable set, and let $S=2^\Xcal$ be the collection of all subsets of~$\Xcal$ partially ordered by set inclusion. Supremum is set union, $-\infty=\emptyset$, and $+\infty=\Xcal$. With these operations $S$ is a complete lattice, and it admits a $\sigma$-algebra $\Scal$ such that \ref{A1}--\ref{A3} are satisfied; see Theorem~\ref{T:countable set}.
\end{example}

\begin{example} \label{E:CI 2}
Let $S={\rm CO}(\R^d)$ be the collection of all closed convex subsets $C\subseteq\R^d$ partially ordered by set inclusion. For a subset $A\subseteq S$ one has $\sup A=\overline\conv (\bigcup_{C\in A}C)$, the closed convex hull of the union of all $C\in A$, and $\inf A=\bigcap_{C\in A}C$. Moreover, $-\infty=\emptyset$ and $+\infty=\R^d$. With these operations $S$ is a complete lattice, and it admits a $\sigma$-algebra $\Scal$ such that \ref{A1}--\ref{A3} are satisfied; see Theorem~\ref{T:convex sets}.
\end{example}

The following lemma is a consequence of the existence of a strictly increasing measurable real-valued map. We will use it to define the conditional infimum.

\begin{lemma} \label{L:esssup}
Let $\Lcal$ be a set of random elements of $S$ closed under countable suprema. Then $\Lcal$ contains a maximal element. That is, there exists $X^*\in\Lcal$ such that $X\le X^*$ almost surely for every $X\in\Lcal$. The maximal element $X^*$ is unique up to almost sure equivalence.
\end{lemma}

\begin{proof}
The uniqueness statement is obvious since any other maximal element $X^{**}\in\Lcal$ satisfies $X^*\le X^{**}\le X^*$ almost surely. To prove existence, let $\phi\colon S\to\R$ be a strictly increasing measurable map, without loss of generality taken to be bounded, and define
\[
\alpha = \sup\{ \E[\phi(X)] \colon X\in\Lcal\}.
\]
Let $(X_n)_{n\in\N}$ be a maximizing sequence and define $X^*=\sup_n X_n \in\Lcal$. Then
\[
\alpha\ge\E[\phi(X^*)]\ge\E[\phi(X_n)] \to \alpha,
\]
so $\E[\phi(X^*)]=\alpha$. Consider any $X\in\Lcal$ and assume for contradiction that $\P(X\not\le X^*)>0$. Then the random element $Y=X^*\vee X \in \Lcal$ satisfies $X^*\le Y$ and $\P(X^*<Y)>0$. Therefore, since $\phi$ is strictly increasing,
\[
\alpha \ge \E[\phi(Y)] > \E[\phi(X^*)] = \alpha,
\]
a contradiction. Thus $X\le X^*$ almost surely, as desired.
\end{proof}

Although it will not be used in this paper, let us mention that Lemma~\ref{L:esssup} implies the existence of essential suprema. Given a set $\Lcal$ of random elements of $S$, a random element $X^*$ is the {\em essential supremum of $\Lcal$} (necessarily a.s.~unique) if $X^*$ a.s.~dominates $\Lcal$ and satisfies $X^*\le Y$ a.s.~for any random element $Y$ that also a.s.~dominates $\Lcal$.

\begin{corollary}
Let $\Lcal$ be any set of random elements of $S$. Then $\Lcal$ admits an essential supremum $X^*$, which can be expressed as the supremum of countably many elements of $\Lcal$.
\end{corollary}

\begin{proof}
Let $\overline\Lcal$ be the set of all countable suprema $\sup_nX_n$ of elements $X_n\in\Lcal$. This set is closed under countable suprema, and thus admits a maximal element $X^*$ by Lemma~\ref{L:esssup}. Moreover, if $Y$ dominates $\Lcal$, it also dominates $\overline\Lcal$, whence $X^*\le Y$. Finally, being an element of $\overline\Lcal$, $X^*$ is the supremum of countably many elements of $\Lcal$.
\end{proof}

The following definition introduces the key object of interest in this paper, the conditional infimum. Lemma~\ref{L:esssup} implies that the the conditional infimum always exists and is unique up to almost sure equivalence.

\begin{definition} \label{D:cond inf} 
Let $X$ be a random element of $S$, and let $\Ecal\subseteq\Fcal$ be a sub-$\sigma$-algebra. The {\em conditional infimum of $X$ given $\Ecal$}, denoted by $\bigwedge[X\mid\Ecal]$, is the maximal element of
\[
\{Z:\Omega\to S \text{ such that $Z$ is $\Ecal$-measurable and $Z\le X$ almost surely} \}.
\]
That is, $\bigwedge[X\mid\Ecal]$ is the greatest $\Ecal$-measurable lower bound on $X$.
\end{definition}

The following lemma collects some basic properties of the conditional infimum, which are immediate consequences of the definition. These properties are well-known in the literature, at least in the case $S=\overline\R$; see \cite{bar_car_jen_03}.

\begin{lemma}[Properties of the conditional infimum] \label{L:PCI}
Let $X$ and $Y$ be random elements of~$S$, and let $\Ecal$ and $\Gcal$ be sub-$\sigma$-algebras of $\Fcal$. Then the following properties hold:
\begin{enumerate}
\item\label{PCI:mon I}If $\Ecal\subseteq\Gcal$ then $\bigwedge[X\mid\Ecal]\le\bigwedge[X\mid\Gcal]$.
\item\label{PCI:mon II}If $X\le Y$ then $\bigwedge[X\mid\Ecal]\le\bigwedge[Y\mid\Ecal]$.
\item\label{PCI:tower}If $\Ecal\subseteq\Gcal$, then $\bigwedge[\,\bigwedge[X\mid\Gcal]\mid\Ecal] = \bigwedge[X\mid\Ecal]$.
\item\label{PCI:cont I}Let $\{\Ecal_n\}_{n\in\N}$ be a non-increasing sequence of sub-$\sigma$-algebras and suppose $\Ecal = \bigcap_{n\in\N}\Ecal_n$. Then $\bigwedge[X\mid\Ecal] = \inf_{n\in\N} \bigwedge[X\mid\Ecal_n]$.
\item\label{PCI:cont II} Let $\{X_n\}_{n\in\N}$ be a sequence of random elements of~$S$. Then $\bigwedge[\inf_{n\in\N}X_n\mid\Ecal] = \inf_{n\in\N} \bigwedge[X_n\mid\Ecal]$.
\item\label{PCI:max linear} If $Y$ is $\Ecal$-measurable and $\le$ is a total order, then $\bigwedge[X\vee Y\mid\Ecal] = \bigwedge[X\mid\Ecal] \vee Y$.
\end{enumerate}
\end{lemma}

\begin{proof}
\ref{PCI:mon I}: $\bigwedge[X\mid\Ecal]$ is $\Gcal$-measurable and dominated by $X$.
\ref{PCI:mon II}: Any lower bound of $X$ is also a lower bound on $Y$.
\ref{PCI:tower}: By monotonicity, $\bigwedge[\,\bigwedge[X\mid\Gcal]\mid\Ecal]\le\bigwedge[X\mid\Ecal]\le X$. Moreover, if $Z\le X$ is $\Ecal$-measurable, then it is also $\Gcal$-measurable, whence $Z\le\bigwedge[X\mid\Gcal]$. Thus $Z=\bigwedge[Z\mid\Ecal]\le \bigwedge[\,\bigwedge[X\mid\Gcal]\mid\Ecal]$.
\ref{PCI:cont I}: Since $\Ecal_n$ is non-increasing in $n$, \ref{PCI:mon I} yields $\inf_{n\in\N} \bigwedge[X\mid\Ecal_n]=\inf_{n\ge m} \bigwedge[X\mid\Ecal_n]$ for each $m$. Thus $\inf_{n\in\N} \bigwedge[X\mid\Ecal_n]$ is $\Ecal$-measurable. Moreover, it dominates any $\Ecal$-measurable~$Z\le X$.
\ref{PCI:cont II}: $\bigwedge[X_n\mid\Ecal]$ is a lower bound on $X_n$, whence $\inf_{n\in\N} \bigwedge[X_n\mid\Ecal]$ is a lower bound on $\inf_{n\in\N}X_n$. Thus $\inf_{n\in\N} \bigwedge[X_n\mid\Ecal]\le \bigwedge[\inf_{n\in\N}X_n\mid\Ecal]$. The reverse inequality follows from~\ref{PCI:mon II}.

\ref{PCI:max linear}: Set $X'=\bigwedge[X\mid\Ecal]$. Then $X'\le X$, hence $X'\vee Y\le X\vee Y$. It remains to pick an arbitrary $\Ecal$-measurable $Z\le X\vee Y$ and show that $Z\le X'\vee Y$. On $\{Y<Z\}$ one has $Y<Z\le X\vee Y$ and hence $X\vee Y=X$. Thus
\[
Z\wedge \chi_{\{Y<Z\}^c} \le (X\vee Y) \wedge \chi_{\{Y<Z\}^c} = X \wedge \chi_{\{Y<Z\}^c} \le X.
\]
The left-hand side is $\Ecal$-measurable, so $Z\wedge \chi_{\{Y<Z\}^c}\le X'$ by definition of $X'$. It follows that $Z\le X'\vee Y$ on $\{Y<Z\}$. Since $\le$ is a total order, $\{Y<Z\}^c=\{Z\le Y\}$, so that $Z\le X'\vee Y$ also on this set. Thus $Z\le X'\vee Y$, as required.
\end{proof}

We now consider $S$-valued stochastic processes $V=(V_t)_{t\ge0}$ adapted to the right-continuous filtration~$\F$. A process $V$ with nondecreasing paths is called {\em right-continuous} if it satisfies
\[
V_t(\omega) = \inf_{s>t} V_s(\omega) \quad\text{for all}\quad (\omega,t)\in\Omega\times\R_+.
\]
This amounts to a slight abuse of terminology, since $S$ need not have any topological structure.

Given a random element $X$, one can consider the family $V=(V_t)_{t\ge0}$ of random variables $V_t=\bigwedge[X\mid\Fcal_t]$. In view of Lemma~\ref{L:PCI}\ref{PCI:mon I}, $V_t$ is non-decreasing in $t$; however, at this stage it is only defined up to a nullset that may depend on $t$. The following result confirms that one can choose a regular version.

\begin{lemma}[Right-continuous version] \label{L:RC version}
Let $X$ be a random element of $S$. Then there exists an adapted nondecreasing right-continuous $S$-valued process $V=(V_t)_{t\ge0}$ such that $V_t=\bigwedge[X\mid\Fcal_t]$ for all $t\ge0$. The process $V$ is unique up to evanescence.
\end{lemma}

\begin{proof}
Fix a dense countable subset $D\subset\R_+$, and let $V^0_t$ be a version of $\bigwedge[X\mid\Fcal_t]$ for each $t\in D$. For each $t\in \R_+$, define
\[
H_t = \bigcap_{u>t} H^0_u, \qquad  H^0_t = \bigcap_{\substack{r,s\in D\\ r<s\le t}} \left\{ \omega\colon V^0_r(\omega) \le V^0_s(\omega) \right\}.
\]
Thus $H_t$ is the set of $\omega$ such that the map $s\mapsto V^0_s(\omega)$ is nondecreasing on $D\cap[0,t+\varepsilon]$ for some $\varepsilon>0$. One has $\P(H_t)=1$ by Lemma~\ref{L:PCI}\ref{PCI:mon I}, as well as $H_t\in\Fcal_t$ by right-continuity of $\F$. Define $V=(V_t)_{t\ge 0}$ by
\[
V_t(\omega) = \begin{cases} \inf_{s\ge t,\,s\in D} V^0_s(\omega) & \omega\in H_t, \\ +\infty & \omega\notin H_t.\end{cases}
\]
It follows that $V$ is adapted, nondecreasing, and right-continuous. Furthermore, Lemma~\ref{L:PCI}\ref{PCI:cont I} and right-continuity of $\F$ yield $V_t=\inf_{s\ge t,\,s\in D} V^0_s = \bigwedge[X\mid F_t]$. The uniqueness statement follows from the almost sure uniqueness of each $V_t$ together with right-continuity.
\end{proof}

\begin{lemma}[Optional stopping] \label{L:optional sampling}
Let $X$ be a random element of $S$ and let $V=(V_t)_{t\ge0}$ be the regular version of $V_t=\bigwedge[X\mid\Fcal_t]$. Then
\[
V_\tau = \bigwedge[X\mid \Fcal_\tau ]
\]
for every stopping time $\tau$.
\end{lemma}

\begin{proof}
It suffices to prove the result for $\tau$ taking finitely many values. Indeed, in the general case one has $\lim_{m\to\infty}\tau_m=\tau$ for some non-increasing sequence of stopping times~$\tau_m$ taking finitely many values. Lemma~\ref{L:PCI}\ref{PCI:cont I} and right-continuity of $\F$ and $V$ then yield the results.

We therefore suppose $\tau=\sum_n t_n\bm 1_{A_n}$ for finitely many distinct $t_n\in[0,\infty]$ and pairwise disjoint sets $A_n\in\Fcal_{t_n}$ forming a partition of $\Omega$. Let $Y$ be any $\Fcal_\tau$-measurable random element of $S$ with $Y\le X$. We must show that $Y\le V_\tau$. To this end, define the random elements
\[
Y_n = \begin{cases} Y &\text{on $A_n$} \\ -\infty &\text{on $A_n^c$} \end{cases}
\]
For any $B\in\Scal$ one has $\{Y_n\in B\} = \left(\{Y\in B\}\cap A_n\right) \cup \left(\{\infty\in B\}\cap A_n^c\right)$. This event lies in $\Fcal_{t_n}$ since $\{Y\in B\}\cap A_n=\{Y\in B\}\cap\{\tau=t_n\}\in\Fcal_{t_n}$ by $\Fcal_\tau$-measurability of $Y$, and since $A_n^c=\{\tau\ne t_n\}\in\Fcal_{t_n}$ due to the fact that $\tau$ is a stopping time. Consequently $Y_n$ is $\Fcal_{t_n}$-measurable and satisfies $Y_n\le Y\le X$, so by definition of the conditional infimum we have $Y_n\le\bigwedge[X\mid\Fcal_{t_n}]=V_{t_n}$. Therefore, $Y = \inf_n (Y_n\vee\chi_{A_n}) \le \inf_n (V_{t_n}\vee\chi_{A_n}) = V_\tau$, as required.
\end{proof}

\begin{example}
It is not true in general that $V_{t-}=\bigwedge[X\mid\Fcal_{t-}]$. For example, suppose $S=\overline\R$. Let $W$ be standard Brownian motion and $\F$ the right-continuous filtration it generates. Set $X=|W_1|$ and let $V$ be the regular version of $V_t=\bigwedge[X\mid\Fcal_t]$. Then $V_t=0$ for all $t<1$, but since $\Fcal_{1-}=\Fcal_1$ one has $\bigwedge[X\mid\Fcal_{1-}]=X>0$.
\end{example}

The following theorem is the main result of this section. It provides equivalent conditions for when a monotone process can be recovered from its final value by taking conditional infima.

\begin{theorem}[Recovery of monotone processes] \label{T:NCR}
Let $U=(U_t)_{t\ge0}$ be an adapted nondecreasing right-continuous $S$-valued process, and define $U_\infty=\sup_{t\ge0}U_t$. The following conditions are equivalent, where the regular version of $\bigwedge[U_\infty\mid\Fcal_t]$ is understood:
\begin{enumerate}
\item\label{T:NCR:1} $U_t = \bigwedge[U_\infty\mid\Fcal_t]$ for all $t\ge0$;
\item\label{T:NCR:2} Any stopping time $\tau$ with $U_\tau<Y$ on $\{\tau<\infty\}$ for some $\Fcal_\tau$-measurable $S$-valued random variable $Y\le U_\infty$ satisfies $\P(\tau<\infty)=0$.
\item\label{T:NCR:3}
$\P(Y\le U_\infty \mid\Fcal_\tau)<1$ holds on $\{U_\tau<+\infty\}$ for every stopping time $\tau$ and every $\Fcal_\tau$-measurable $S$-valued random variable $Y$ with $U_\tau<Y$ on $\{U_\tau<+\infty\}$.
\end{enumerate}
\end{theorem}

Condition~\ref{T:NCR:2} of Theorem~\ref{T:NCR} excludes sure improvements. Indeed, if the condition fails for some stopping time $\tau$, then on the positive probability event $\{\tau<\infty\}$, one has $U_\tau <Y\le U_\infty$, where $Y$ is $\Fcal_\tau$-measurable. At time $\tau$ one is therefore guaranteed that $U$ will increase in the future by an amount that is ``bounded away from zero''. In economic terms, supposing $U$ is real-valued to fix ideas, one can think of a situation where exchanging the current value $U_\tau$ for the final outcome $U_\infty$ is guaranteed to result in an $\Fcal_\tau$-measurable gain of at least $Y-U_\tau>0$. On the other hand, if condition~\ref{T:NCR:2} is satisfied, then there is no nontrivial $\Fcal_\tau$-measurable lower bound on the gain. In this sense, \ref{T:NCR:2} is reminiscent of the no-arbitrage type conditions appearing in mathematical finance. This analogy is brought further by the equivalent characterization~\ref{T:NCR:1}, which can be thought of as a martingale condition.

In contrast to the correspondence between no arbitrage and martingales however, Theorem~\ref{T:NCR} does not involve any equivalent changes of probability measure. This is because the conditional infimum only depends on the probability measure through its nullsets, which carries over to the ``martingale'' condition~\ref{T:NCR:1}. Both~\ref{T:NCR:1} and~\ref{T:NCR:2} are thus invariant with respect to equivalent measure changes.

The equivalent property \ref{T:NCR:3} is similar to~\ref{T:NCR:2}, but looks more convoluted. The reason for stating it is that it is closely related to the notion of {\em stickiness} for real-valued increasing processes. In fact, \ref{T:NCR:3} may be viewed as a natural generalization of the stickiness property to processes on $[0,\infty)$ with values in a lattice $S$ which satisfies the assumptions~\ref{A1}--\ref{A3}. Sticky processes are discussed in Section~\ref{S:sticky}.

\begin{proof}[Proof of Theorem~\ref{T:NCR}]
\ref{T:NCR:1} $\Rightarrow$ \ref{T:NCR:2}: Pick a stopping time $\tau$ and an $\Fcal_\tau$-measurable random variable $Y\le U_\infty$ such that $U_\tau<Y$ on $\{\tau<\infty\}$. In particular, $Y\le\bigwedge[U_\infty\mid\Fcal_\tau]$. Together with~\ref{T:NCR:1} and Lemma~\ref{L:optional sampling}, this yields
\[
\bigwedge[U_\infty\mid\Fcal_\tau] = U_\tau < Y \le \bigwedge[U_\infty\mid\Fcal_\tau]
\]
on $\{\tau<\infty\}$. Thus $\P(\tau<\infty)=0$ as required, showing that~\ref{T:NCR:2} holds.

\ref{T:NCR:2} $\Rightarrow$ \ref{T:NCR:3}: Pick a stopping time $\tau$ and an $\Fcal_\tau$-measurable random variable $Y$ with $Y\le U_\infty$ and $U_\tau<Y$ on $\{U_\tau<\infty\}$. Define $A=\{\P(Y\le U_\infty\mid\Fcal_\tau)=1\}\cap\{U_\tau<\infty\}$. We must show that $\P(A)=0$. To this end, define the stopping time
\[
\sigma = \tau\bm 1_A + \infty \bm 1_{A^c}
\]
and the $\Fcal_\sigma$-measurable random variable
\[
Z = (Y\vee \chi_A)\wedge (U_t\vee \chi_{A^c}) = \begin{cases} Y & \text{on $A$}\\ U_t &\text{on $A^c$.}\end{cases}
\]
Since $Y\le U_\infty$ on $A$, we have $Z\le U_\infty$. Moreover, since $\{\sigma<\infty\}\subseteq A$, we have $U_\tau<Z$ on $\{\sigma<\infty\}$. Thus~\ref{T:NCR:2} implies $\P(\sigma<\infty)=0$. Therefore, using also that $A\cap\{\tau=\infty\}$ is a nullset since $Y\le U_\infty=U_\tau<Y$ there, we obtain
\[
\P(A) = \P(A\cap\{\tau<\infty\}) = \P(\sigma<\infty) = 0,
\]
as required.

\ref{T:NCR:3} $\Rightarrow$ \ref{T:NCR:1}: Pick $t\ge0$ and let $A=\{U_t<\bigwedge[U_\infty\mid\Fcal_t]\}$. Define $Y=\bigwedge[U_\infty\mid\Fcal_t] \vee \chi_{A^c}$. Then $Y$ is $\Fcal_t$-measurable, with $U_t<Y$ on $A$ and $Y=+\infty$ on $A^c$, hence $U_t<Y$ on $\{U_t<+\infty\}$. Moreover, $Y\le U_\infty$ on $A$, hence $\P(Y\le U_\infty\mid\Fcal_t)=1$ on $A\subseteq\{U_t<+\infty\}$. By \ref{T:NCR:3}, this forces $\P(A)=0$. Thus~\ref{T:NCR:1} holds.
\end{proof}

\section{Sticky processes} \label{S:sticky}

In this section we apply the theory of Section~\ref{S:CI} with the complete lattice $\overline\R=[-\infty,\infty]$ in Example~\ref{E:CI 1}. We are thus dealing with scalar (i.e., extended real-valued) non-decreasing processes and conditional infima of scalar random variables. In this context there is a close connection between condition~\ref{T:NCR:3} of Theorem~\ref{T:NCR} and the notion of {\em stickiness}.

Throughout this section, $X$ denotes a c\`adl\`ag adapted process with values in a given separable metric space $(\Xcal,d)$, which is equipped with its Borel $\sigma$-algebra.

\begin{definition} \label{D:sticky}
We call $X$ {\em globally sticky} if $\P(\sup_{t\in[\tau,\infty)}d(X_t,X_\tau) \le \varepsilon \mid \Fcal_\tau) > 0$ for every stopping time $\tau$ and every strictly positive $\Fcal_\tau$-measurable random variable $\varepsilon$. We call $X$ {\em sticky} if the stopped process $X^T=(X_{t\wedge T})_{t\ge0}$ is globally sticky for every $T\in\R_+$.
\end{definition}

Note that on the event $\{\tau=\infty\}$, the supremum in this definition is taken over the empty set, and therefore equals $-\infty$ by convention. Thus the conditional probability is equal to one on this event. Furthermore, we never have to evaluate the possibly undefined quantity~$X_\infty$.

\begin{remark}
The terminology of Definition~\ref{D:sticky} is consistent with the existing literature, where stickiness is generally defined for process on a bounded time interval $[0,T]$. In our setting it is more natural to work with process on $[0,\infty)$, which makes the notion of global stickiness convenient.
\end{remark}


A wide variety of processes are sticky. For example, any one-dimensional regular strong Markov process is sticky, see \citet[Proposition~3.1]{G:06}. Moreover, any process with conditional full support is sticky; see \citet{GRS:08,BPS:15}. This includes most L\'evy processes, large classes of solutions of stochastic differential equations, processes like skew Brownian motion, as well as non-Markovian non-semimartingales like fractional Brownian motion. We will return to the conditional full support property in connection with Proposition~\ref{P:int f sticky} below. Continuous functions of sticky processes are sticky, and stickiness is preserved under bounded time changes. \citet{RS:16} provide further examples and references, and we develop some additional results in this direction below.

For a non-decreasing $\R$-valued process $U$, global stickiness reduces to the condition
\begin{equation} \label{eq:U sticky}
\P(U_\infty \le U_\tau + \varepsilon \mid \Fcal_\tau) > 0 \text{ for any stopping time $\tau$ and $\Fcal_\tau$-measurable $\varepsilon>0$,}
\end{equation}
where $U_\infty=\lim_{t\to\infty}U_t$. This immediately yields the following corollary of Theorem~\ref{T:NCR}, which explains the relevance of stickiness in the present context.

\begin{corollary} \label{C:sticky}
An adapted non-decreasing right-continuous $\R$-valued process $U$ satisfies $U_t=\bigwedge[U_\infty\mid\Fcal_t]$ for all $t\ge0$ if and only if it is globally sticky.
\end{corollary}

\begin{proof}
View $U$ as taking values in $S=\overline\R$. By Theorem~\ref{T:NCR}, the equality $U_t=\bigwedge[U_\infty\mid\Fcal_t]$ holds for all $t\in\R_+$ if and only if $\P(U_\tau + \varepsilon \le U_\infty\mid\Fcal_\tau)<1$ holds on $\{U_\tau<\infty\}$ for every stopping time $\tau$ and every $\Fcal_\tau$-measurable $\varepsilon>0$. Applying this with~$\varepsilon/2$ in place of~$\varepsilon$, one sees that the weak inequality can be replaced by a strict inequality. Therefore the inequality $\P(U_\tau + \varepsilon \le U_\infty\mid\Fcal_\tau)<1$ can be replaced by $\P(U_\infty \le U_\tau + \varepsilon \mid\Fcal_\tau)>0$. Consequently, since $U$ is actually finite-valued, the above statement is equivalent to the stickiness property~\eqref{eq:U sticky}.
\end{proof}

\begin{remark} Inspired by Corollary~\ref{C:sticky}, one may use condition \ref{T:NCR:3} of Theorem~\ref{T:NCR} to {\em define} global stickiness for nondecreasing processes valued in a complete lattice satisfying the assumptions~\ref{A1}--\ref{A3}.
\end{remark}

Corollary~\ref{C:sticky} is useful because non-decreasing functionals of sticky processes are often sticky, which means that there is an abundance of non-decreasing sticky processes. We now provide a number of results in this direction.

\begin{proposition} \label{P:f(X) run max}
Let $U_t=\sup_{s\le t}f(X_s)$, where $f:E\to\R$ is a continuous map. If $X$ is (globally) sticky, then $U$ is also (globally) sticky.
\end{proposition}

\begin{proof}
Assume that $X$ is globally sticky. The result for $X$ sticky but not globally sticky follows by replacing $X$ by $X^T$ for any $T<\infty$ in the argument below. Fix any stopping time $\tau$ and $\Fcal_\tau$-measurable random variable $\varepsilon>0$. On the event $\{\tau<\infty\}$, the random set $C=f^{-1}((-\infty,f(X_\tau)+\varepsilon))$ is open and contains $X_\tau$. One can find a strictly positive $\Fcal_\tau$-measurable random variable $\varepsilon'$ such that, on $\{\tau<\infty\}$, $C$ contains the closed ball of radius $\varepsilon'$ centered at $X_\tau$. Consequently,
\begin{align*}
\P(U_\infty \le U_\tau+\varepsilon \mid \Fcal_\tau) &\ge\P(f(X_t) \le f(X_\tau) + \varepsilon \text{ for all $t\in[\tau,\infty)$} \mid \Fcal_\tau) \\
&\ge \P( d(X_t,X_\tau) \le \varepsilon' \text{ for all $t\in[\tau,\infty)$} \mid\Fcal_\tau) \\
&>0,
\end{align*}
using that $X$ is globally sticky. Thus $U$ is also globally sticky.
\end{proof}

\begin{corollary}
If $X$ is real-valued and (globally) sticky, then $\overline X_t=\max_{0\le s\le t}X_s$ and $X^*_t=\max_{0\le s\le t}|X_s|$ are also (globally) sticky.
\end{corollary}

The next result looks somewhat abstract, but has useful consequences. In particular, it implies that the local time of a sticky semimartingale is again sticky; see Corollary~\ref{C:sticky local time} below. We let $d(x,K)=\inf\{d(x,y)\colon y\in K\}$ denote the distance from a point $x\in \Xcal$ to a subset $K\subseteq\Xcal$.

\begin{proposition} \label{P:X U K sticky}
Let $K\subseteq \Xcal$ be a closed subset, and let $U$ be a nondecreasing right-continuous adapted process that satisfies the following property for almost every~$\omega$:
\begin{equation}\label{P:X U K sticky cond}
\text{\parbox{0.8\textwidth}{If $[t_1,t_2]$ is an interval such that either $X(\omega)\in K$ on $[t_1,t_2)$, or\\ $d(X(\omega),K)\ge a$ on $[t_1,t_2]$ for some $a>0$, then $U$ is constant on $[t_1,t_2]$.}}
\end{equation}
If $X$ is (globally) sticky, then $U$ is also (globally) sticky.
\end{proposition}

\begin{proof}
We prove the result for $X$ globally sticky. Fix any stopping time $\tau$ and any $\Fcal_\tau$-measurable $\varepsilon>0$. For each $a>0$, define the stopping time $\sigma_a=\inf\{t\ge \tau\colon d(X_t,K)\ge a\}$. Since $U$ satisfies \eqref{P:X U K sticky cond}, the equality $U_\infty=U_{\sigma_a}$ holds on the event where $d(X_t,K) \ge a/2$ for all $t\in[\sigma_a,\infty)$. Consequently, for any $a>0$,
\begin{align*}
\P(U_\infty \le U_\tau + \varepsilon \mid\Fcal_\tau)
&\ge \P(U_{\sigma_a} \le U_\tau + \varepsilon \text{ and } d(X_t,K) \ge \frac{a}{2} \text{ for all $t\in[\sigma_a,\infty)$} \mid\Fcal_\tau) \\
&= \E\left[ \bm 1_{\{U_{\sigma_a} \le U_\tau + \varepsilon\}}\, \P( d(X_t,K) \ge \frac{a}{2} \text{ for all $t\in[\sigma_a,\infty)$} \mid \Fcal_{\sigma_a}) \Mid \Fcal_\tau \right].
\end{align*}
Consider now the event
\[
A = \{\P(U_\infty \le U_\tau + \varepsilon \mid\Fcal_\tau) = 0\} \in \Fcal_\tau.
\]
Then, by the inequality above,
\[
\bm 1_{\{U_{\sigma_a} \le U_\tau + \varepsilon\}}\, \P( d(X_t,K) \ge a/2 \text{ for all $t\in[\sigma_a,\infty)$} \mid \Fcal_{\sigma_a})=0
\]
holds on~$A$ for all $a>0$. The global stickiness property states that the conditional probability is strictly positive, whence $\bm 1_{\{U_{\sigma_a} \le U_\tau + \varepsilon\}}=0$ on $A$ for all $a>0$. Define the stopping time $\sigma_0=\inf\{t\ge \tau\colon X_t\notin K\}$. Sending $a$ to zero and using that $K$ is closed, we obtain $\sigma_a\downarrow \sigma_0$. Right-continuity and non-decrease of $U$ then yields $U_{\sigma_a}\downarrow U_{\sigma_0}$, and hence
\begin{equation}\label{P:X U K sticky:eq2}
\bm 1_{\{U_{\sigma_0} < U_\tau + \varepsilon\}} \le \lim_{a\downarrow0}\bm 1_{\{U_{\sigma_a} \le U_\tau + \varepsilon\}}=0 \quad\text{on $A$.}
\end{equation}
But since $X$ lies in $K$ on $[\tau,\sigma_0)$, the assumption~\eqref{P:X U K sticky cond} on $U$ implies that $U_{\sigma_0}=U_\tau$. Thus the left-hand side of~\eqref{P:X U K sticky:eq2} equals one, which forces $\P(A)=0$. Thus $\P(U_\infty \le U_\tau + \varepsilon \mid\Fcal_\tau)>0$, that is, $U$ is sticky.
\end{proof}

\begin{corollary} \label{C:sticky local time}
Suppose $X$ is a real semimartingale, and let $L^x$ be its local time at level $x$. If $X$ is (globally) sticky, and $x_1,\ldots,x_n\in\R$, $n\in\N$ are arbitrary, then $L^{x_1}+\cdots+L^{x_n}$ is also (globally) sticky.
\end{corollary}

\begin{proof}
Apply Proposition~\ref{P:X U K sticky} with $K=\{x_1,\ldots,x_n\}$ and $U=L^{x_1}+\cdots+L^{x_n}$.
\end{proof}

\begin{example}
Without the stickiness assumption on $X$, there is no guarantee that its local time is sticky. Indeed, let $W$ be Brownian motion, which is not globally sticky. Its local time $L^0_t(W)$ tends to infinity with $t$, so that $\bigwedge[L^0_\infty\mid\Fcal_t]=\infty\ne L^0_t(W)$. This can be turned into an example on a finite time horizon as follows. Let $\tau=\inf\{t\ge0\colon L^0_t(W)\ge 1\}$, which is a finite stopping time. Define $X_t = W_{(t/(1-t))\wedge\tau}$ for $t\in[0,1]$, which for $t=1$ should be read $X_1=W_\tau$. Then $X$ is a semimartingale with respect to the time-changed filtration, and its local time is given by $L^0_t(X)=L^0_{(t/(1-t))\wedge\tau}(W)$. In particular, one has $\bigwedge[L^0_1(X)]=1\ne 0=L^0_0(X)$.
\end{example}

For the next result we assume that $\Xcal$ is an open connected subset of $\R^n$ and that $X$ has continuous paths. For any deterministic times $\tau\le T<\infty$, the restriction $X|_{[\tau,T]}=(X_t)_{t\in[\tau,T]}$ is a random element of the space $C([\tau,T];\Xcal)$ of all continuous maps $[\tau,T]\to \Xcal$, equipped with the uniform metric. The process $X$ is said to have {\em conditional full support} if for any choice of deterministic times $0\le\tau<T$, the support of the regular conditional distribution of $X|_{[\tau,T]}$ given $\Fcal_\tau$ is almost surely all of $C_{X_\tau}([\tau,T];\Xcal)$, the closed subset of $C([\tau,T];\Xcal)$ whose paths are equal to $X_\tau$ at time $\tau$. The notion of conditional full support plays an important role in mathematical finance, and implies the stickiness property; see e.g.~\citet{BPS:15}.

\begin{proposition} \label{P:int f sticky}
Let $f:\Xcal\to\R_+$ be a nonnegative continuous function with $0\in f(\Xcal)$. If $X$ has conditional full support, then the process $U$ given by
\[
U_t = \int_0^t f(X_s) ds
\]
is also sticky.
\end{proposition}

\begin{proof}
We must show that for any $T<\infty$, any stopping time $\tau\le T$, and any strictly positive $\Fcal_\tau$-measurable $\varepsilon$, we have $\P(U_T\le U_\tau+\varepsilon\mid\Fcal_\tau)>0$. By \citet[Lemma~3.1]{BPS:15}, it suffices to take $\tau<T$ deterministic; see also \citet[Lemma~2.1]{RS:16}. Consider the regular conditional distribution of $(\varepsilon,X|_{[\tau,T]})$ given $\Fcal_\tau$, under which $X_\tau$ and $\varepsilon$ are both Dirac distributed almost surely and therefore can be treated as being deterministic. Let $\gamma\colon [0,1]\to \Xcal$ be a continuous map with $\gamma(0)=X_\tau$ and $f(\gamma(1))=0$. Such a map exists since $\Xcal$ is connected and $0\in f(\Xcal)$. Let $m=\max_{s\in[0,1]}f(\gamma(s))$ be the largest value that $f$ attains along $\gamma$.

Now define the set $G\subset C_{X_\tau}([\tau,T];\Xcal)$ to consist of all $\Xcal$-valued continuous paths $x\colon [\tau,T]\to \Xcal$ with $x(\tau)=X_\tau$ satisfying the following two properties:
\begin{enumerate}
\item $f(x(t)) < m+\varepsilon$ for all $t\in[\tau,T]$.
\item $f(x(t)) < \varepsilon/(2T)$ for all $t\in[\sigma,T]$, where we define $\sigma\in(\tau,T)$ by
\[
\sigma=\tau + \frac{\varepsilon/2}{m + \varepsilon} \wedge \frac{T-\tau}{2}.
\]
\end{enumerate}
Then $G$ is open, being the intersection of two open sets. Moreover, $G$ is nonempty since it contains the path $(x(t))_{t\in[\tau,T]}$ given by
\[
x(t) = \gamma\left( 1\wedge\frac{t-\tau}{\sigma-\tau}\right).
\]
The conditional full support property therefore implies that the event $\{X|_{[\tau,T]}\in G\}$ has strictly positive regular conditional probability. On the other hand, whenever $X|_{[\tau,T]}$ remains in $G$ one has
\[
U_T - U_\tau = \int_\tau^\sigma f(X_t) dt + \int_\sigma^T f(X_t)dt \le \frac{\varepsilon/2}{m+\varepsilon}\times(m+\varepsilon) + T\times \frac{\varepsilon}{2T} = \varepsilon.
\]
In conclusion, one has $\P(U_T-U_\tau \le \varepsilon\mid\Fcal_\tau)>0$ as required.
\end{proof}

\section{Martingales and supermartingales} \label{S:martingales}

Martingales are not always sticky: one example is the martingale $M_t=\P(W_1>0\mid\Fcal_t)$ where $W$ is Brownian motion. This martingale starts at $1/2$ and converges to zero or one at time $t=1$. Therefore it does not remain in any interval around $1/2$ of width strictly less than $1/2$. Nonetheless, certain functionals of martingales are necessarily sticky, and consequently satisfy the ``inf-martingale'' property~\ref{T:NCR:1} of Theorem~\ref{T:NCR}. This includes the running maximum process $\overline M_t=\sup_{s\le t}M_s$ as well as the local time processes.

The following result is an immediate consequence of Theorem~\ref{T:NCR}. Recall that a c\`adl\`ag supermartingale $M$ is {\em closed on the right} if there exists an integrable random variable $X$ such that $M_t\ge \E[X\mid\Fcal_t]$ for all $t\ge0$. For instance, this holds if $M$ is nonnegative or, more generally, if $\{M_t^-\colon t\ge0\}$ is a uniformly integrable family; see VI.8 in~\citet{DM:82}.

\begin{proposition} \label{P:maxM}
Let $M$ be a c\`adl\`ag supermartingale, possibly after an equivalent change of probability measure. Then $\overline M$ is sticky, that is,
\[
\overline M_t = \bigwedge[ \overline M_T \mid\Fcal_T], \qquad t\le T<\infty.
\]
If additionally $M$ is closed on the right, then $\overline M$ is globally sticky, that is,
\[
\overline M_t = \bigwedge[ \overline M_\infty \mid\Fcal_t], \qquad t\ge0.
\]
\end{proposition}

\begin{proof}
We apply the theory of Section~\ref{S:CI} with the complete lattice $\overline\R=[-\infty,\infty]$ in Example~\ref{E:CI 1}. Since the desired conclusion is invariant under equivalent changes of probability measure, we may suppose $M$ is already a supermartingale. We may also suppose it is closed on the right, since we otherwise replace $M$ by $M^T$. The result now follows from Theorem~\ref{T:NCR} with $U=\overline M$, once condition~\ref{T:NCR:2} of the theorem is verified. Thus, consider any stopping time~$\tau$ such that $\overline M_\tau < Y$ on $\{\tau<\infty\}$ for some $\Fcal_\tau$-measurable random variable $Y \le \overline M_\infty$. Define $\sigma=\inf\{t>\tau\colon M_t\ge Y\}$. Then
\[
0 \ge \E[ M_\sigma - M_\tau ] = \E[ (M_\sigma - M_\tau)\bm 1_{\{\tau<\infty\}} ] \ge \E[ (Y - \overline M_\tau)\bm 1_{\{\tau<\infty\}} ] \ge 0.
\]
Therefore $\E[ (Y - \overline M_\tau)\bm 1_{\{\tau<\infty\}} ]=0$, and we deduce $\P(\tau<\infty)=0$, as required.
\end{proof}

An interesting consequence of Proposition~\ref{P:maxM} is that it allows to reconstruct any nonnegative local martingale $M$ from the pair $(M_\infty,\overline M_\infty)$. For uniformly integrable martingales this is obvious, since $M_t=\E[M_\infty\mid\Fcal_t]$ for all $t\ge0$. For general nonnegative local martingales the result is less obvious and even counterintuitive (at least to the author); in particular, many such local martingales satisfy $M_\infty=0$, in which case the global maximum $\overline M_\infty$ alone contains the same information as the entire process.

To reconstruct $M$ from $(M_\infty,\overline M_\infty)$, simply observe that a reducing sequence for $M$ is given by the crossing times $\tau_n=\inf\{t\ge0:\overline M_t\ge n\}$, so that
\[
M_{t\wedge\tau_n} = \E[M_{\tau_n}\mid\Fcal_t] = \E[\overline M_{\tau_n}\bm 1_{\{\tau_n<\infty\}} + M_\infty\bm 1_{\{\tau_n=\infty\}}\mid\Fcal_t].
\]
Thus $M_t=\lim_{n\to\infty}M_{t\wedge\tau_n}$ is determined by $(M_\infty,\overline M)$, which by Proposition~\ref{P:maxM} is determined by $(M_\infty,\overline M_\infty)$.

In fact, a stronger statement is true: it is enough to know only the very largest values of $\overline M_\infty$, in the following sense.

\begin{proposition} \label{P:maxM X}
Let $M$ be a nonnegative local martingale and let $X$ be any bounded random variable. Then $M$ can be reconstructed from the pair $(M_\infty, X\vee \overline M_\infty)$.
\end{proposition}

\begin{proof}
Define $V_t=\bigwedge[X\vee \overline M_\infty\mid\Fcal_t]$ and $\tau_n=\inf\{t\ge0\colon V_t\ge n\}$. Let $c\ge X$ be a deterministic upper bound on $X$. We claim that $\overline M_t=V_t$ on $A=\{V_t\ge n\}$ for any $n>c$ and any $t\ge0$. To see this, note that $X<V_t\le X\vee\overline M_\infty$ on $A$ and hence $X<\overline M_\infty$ on $A$. Thus by Lemma~\ref{L:PCI}\ref{PCI:max linear} and Proposition~\ref{P:maxM},
\[
V_t \vee\chi_A = \bigwedge[X\vee \overline M_\infty \vee \chi_A \mid\Fcal_t] = \bigwedge[\overline M_\infty \vee \chi_A \mid\Fcal_t] = \bigwedge[\overline M_\infty \mid\Fcal_t] \vee \chi_A = \overline M_t\vee\chi_A.
\]
This proves that $\overline M_t=V_t$ on $\{V_t\ge n\}$, as claimed. In conjunction with the inequality $\overline M_t\le V_t$, this implies that $\tau_n=\inf\{t\ge0:\overline M_t\ge n\}$ and $V_{\tau_n}=\overline M_{\tau_n}$ on $\{\tau_n<\infty\}$ for all $n>c$. The argument preceding the theorem now yields the desired result.
\end{proof}

The fact that a nonnegative local martingale $M$ can be reconstructed from the pair $(M_\infty, \overline M_\infty)$ can be deduced from results that already exist in the literature, under the additional assumption that $\overline M$ is continuous. For example, assuming without loss of generality that $M_\infty=0$, a conditional version of an argument by \citet{ELY:97} shows that
\begin{equation} \label{eq:ELYeq}
M_t=\lim_{n\to\infty}n\,\P(\overline M_\infty\ge n\mid\Fcal_t).
\end{equation}
An alternative argument is based on the following identity due to \citet{NY:06}, where it is additionally assumed that $M_0=1$ and $M>0$:
\begin{equation} \label{eq:NYeq}
\E[f(\overline M_\infty)\mid\Fcal_t] =  f(\overline M_t)\left(1 - \frac{M_t}{\overline M_t}\right) + M_t \int_{\overline M_t}^\infty \frac{f(x)}{x^2}dx
\end{equation}
for any positive or bounded Borel function $f$. Choosing for $f$ functions $f_n$ such that $f_n=0$ on $(-\infty,n]$ and $\int_n^\infty x^{-2}f_n(x)dx=1$, the right-hand side of~\eqref{eq:NYeq} becomes equal to $M_t$ as soon as $n$ exceeds $\overline M_t$. This shows that
\[
M_t = \lim_{n\to\infty} \E[f_n(\overline M_\infty)\mid\Fcal_t],
\]
which shows that $M_t$ can be recovered from $\overline M_\infty$.

Note that \eqref{eq:ELYeq} crucially relies on the assumption that $\overline M$ is continuous. Indeed, \citet[Example~3.2]{HR:15} construct a nonnegative martingale $M$, with very large but unlikely upward jumps, such that $M_0=1$, $M_\infty=0$, and
\[
\lim_{n\to\infty} n\,\P(\overline M_\infty \ge n) = 0.
\]
This is inconsistent with~\eqref{eq:ELYeq}. The continuity of $\overline M$ is similarly crucial for \eqref{eq:NYeq}.

Our next result shows that another interesting functional, namely the local time process of a local martingale, is always sticky.

\begin{proposition}
Let $M$ be a local martingale, and let $L^x$ denote its local time at level $x$. Then $L^x$ is sticky, that is,
\[
L^x_t = \bigwedge[L^x_T\mid\Fcal_t], \qquad t\le T<\infty.
\]
\end{proposition}

\begin{proof}
By localization we may assume that $M$ is a martingale. Pick any $T<\infty$, any stopping time $\tau\le T$, and any strictly positive $\Fcal_\tau$-measurable random variable $\varepsilon$. To verify the stickiness property~\eqref{eq:U sticky}, we must show that $\P(L^x_T-L^x_\tau \le 2\varepsilon\mid\Fcal_\tau)>0$. To this end, define stopping times
\begin{align*}
\rho_n&=\inf\{t\ge\tau\colon |M_t-x|\ge n^{-1}\}\wedge T, \\
\rho&=\inf\{t\ge\tau\colon M_t\ne x\}\wedge T.
\end{align*}
We first show that
\begin{equation} \label{eq:M LT 1}
\P(L^x_T - L^x_{\rho_n} \le \varepsilon \mid\Fcal_{\rho_n}) >0.
\end{equation}
Let $B=\{\P(L^x_T - L^x_{\rho_n} \le \varepsilon \mid\Fcal_{\rho_n}) = 0\}$ and define the stopping time
\[
\upsilon = \inf\{t\ge\rho_n\colon M_t = x\} \wedge T.
\]
On $B$ we know that the local time process increases over the interval $[\rho_n,T]$ (in fact, it increases by more than $\varepsilon$). By \citet[Theorem~IV.7]{P:05}, the local time measure $dL_t$ is concentrated on those time points $t$ for which $M_{t-}=M_t=x$. Therefore $M_\upsilon=x$ on $B$. Moreover, $\rho_n$ occurs strictly before $T$ on $B$, so that $|M_{\rho_n}-x|\ge n^{-1}$ on $B$. Combining these observations yields
\[
\E[M_\upsilon - M_{\rho_n} \mid \Fcal_{\rho_n}]
= \begin{cases}
\E[x-M_{\rho_n}\mid\Fcal_{\rho_n}] \le -{\displaystyle\frac{1}{n}} & \text{on $\{M_{\rho_n}\ge x+n^{-1}\}\cap B$}, \\[1ex]
\E[x-M_{\rho_n}\mid\Fcal_{\rho_n}] \ge {\displaystyle\frac{1}{n}} & \text{on $\{M_{\rho_n}\le x-n^{-1}\}\cap B$}.
\end{cases}
\]
Thus $\left|\E[M_\upsilon - M_{\rho_n} \mid \Fcal_{\rho_n}] \right| \ge n^{-1}$ on $B$. The martingale property then forces $\P(B)=0$, which proves~\eqref{eq:M LT 1}.

Next, we prove that
\begin{equation} \label{eq:M LT 2}
\P(L^x_T - L^x_\rho \le 2\varepsilon \mid\Fcal_\rho) >0.
\end{equation}
To this end, define the stopping time
\[
\sigma = \inf\{t\ge \rho\colon L^x_t \ge L^x_\rho + \varepsilon\}.
\]
On the event $\{\sigma>T\}$, clearly $L^x_T-L^x_\rho\le\varepsilon$. On the event $\{\rho_n\le\sigma\le T\}$, one has $L^x_\sigma\ge L^x_{\rho_n}$ and $L^x_\sigma-L^x_\rho=\varepsilon$, hence $L^x_T-L^x_\rho\ge L^x_T-L^x_{\rho_n} + \varepsilon$. Consequently, for each $n$,
\begin{align*}
\P(L^x_T - L^x_\rho \le 2\varepsilon \mid\Fcal_\rho)
&\ge \P(L^x_T - L^x_{\rho_n} \le \varepsilon, \ \rho_n\le \sigma \mid\Fcal_\rho) \\
&= \E[\bm 1_{\{\rho_n\le \sigma\}} \P(L^x_T - L^x_{\rho_n} \le \varepsilon \mid\Fcal_{\rho_n}) \mid \Fcal_\rho].
\end{align*}
Let $A=\{\P(L^x_T - L^x_\rho \le 2\varepsilon \mid\Fcal_\rho)=0\}$. The above inequality along with~\eqref{eq:M LT 1} yields
\[
\bm 1_{\{\rho_n\le \sigma\}\cap A} = 0
\]
for all $n$. Since $\rho_n\downarrow \rho$, and since $\sigma>\rho$, it follows that $\P(A)=0$. This proves~\eqref{eq:M LT 2}.

Finally, just observe that $M$ is constant and equal to $x$ on $[\tau,\rho)$, so that $L^x_\tau=L^x_\rho$. Therefore
\[
\P(L^x_T-L^x_\tau \le 2\varepsilon\mid\Fcal_\tau) = \E[ \P(L^x_T-L^x_\rho \le 2\varepsilon\mid\Fcal_\rho)\mid\Fcal_\tau] > 0
\]
due to~\eqref{eq:M LT 2}. This completes the proof.
\end{proof}

\section{Further examples of recovery of monotone processes} \label{S:further}

We now consider two examples of set-valued nondecreasing processes that can be recovered from their final values. The first example deals with convex hulls, and we apply the theory of Section~\ref{S:CI} with the complete lattice $S={\rm CO}(\R^d)$ in Example~\ref{E:CI 2}. The second example deals with the collection of sites visited by a random walk on a countable set $\Xcal$, and uses the complete lattice $S=2^\Xcal$ in Example~\ref{E:CI 2_new}.

\subsection{Convex hulls}

Let $X=(X_t)_{t\ge0}$ be a c\`adl\`ag adapted process with values in $\R^d$. By Lemma~\ref{L:conv is adapted}, the ${\rm CO}(\R^d)$-valued process $U=(U_t)_{t\ge0}$ given by
\[
U_t = \overline\conv(X_s\colon s\le t)
\]
is adapted. We have the following result.

\begin{proposition} \label{P:convex hull sticky}
If $X$ is sticky, then
\[
U_t = \bigwedge [U_T \mid \Fcal_t], \qquad t\le T<\infty.
\]
\end{proposition}

\begin{proof}
Relying on the implication $\ref{T:NCR:3}\Rightarrow\ref{T:NCR:1}$ of Theorem~\ref{T:NCR}, it suffices to consider any stopping time $\tau\le T$ and $\Fcal_\tau$-measurable ${\rm CO}(\R^d)$-valued random variable $Y$ such that $U_\tau\subsetneq Y$, and prove that $\P(Y\subseteq U_T\mid\Fcal_\tau)<1$. Define the $\Fcal_\tau$-measurable random variable
\[
\varepsilon = 1\wedge \frac12 \sup_{y\in Y} \inf_{x\in U_\tau} |x-y|
\]
which is strictly positive since $U_\tau\subsetneq Y$. Furthermore, one has
\[
Y \not\subseteq \overline\conv( U_\tau \cup B(X_\tau,\varepsilon)),
\]
where $B(X_\tau,\varepsilon)$ is the ball of radius $\varepsilon$ centered at $X_\tau$. Since $X$ is sticky, one therefore gets
\begin{align*}
0 &< \P(\sup_{t\in[\tau,T]} |X_t-X_\tau| \le \varepsilon \mid\Fcal_\tau) \\
&\le \P( U_T \subseteq \overline\conv(U_\tau \cup B(X_\tau,\varepsilon) \mid\Fcal_\tau) \\
&\le \P( Y \not\subseteq U_T \mid\Fcal_\tau).
\end{align*}
This yields $\P(Y\subseteq U_T\mid\Fcal_\tau)<1$ as required.
\end{proof}

\subsection{Sites visited by a random walk}

Let $X=(X_t)_{t\ge0}$ be a c\`adl\`ag process with values in a countable set $\Xcal$. Define the $2^\Xcal$-valued process $U=(U_t)_{t\ge0}$ by
\[
U_t = \bigcup_{s\le t}\{X_s\}.
\]
This is the process whose value at time $t$ is the set of all sites $X$ has visited up to and including time $t$, and is adapted by Lemma~\ref{L:range of SP}. In this context, if we equip $\Xcal$ with the discrete metric $d(x,y)=\bm1_{\{y\}}(x)$, stickiness of $X$ simply means that
\[
\P(X_t = X_\tau \text{ for all $t\in[\tau,T]$} \mid\Fcal_\tau) > 0
\]
for every $T\ge0$ and every stopping time $\tau\le T$. That is, $X$ has conditionally unbounded holding times.

\begin{proposition} \label{P:RV sticky}
Assume $X$ has conditionally unbounded holding times in the above sense. Then
\[
U_t = \bigwedge [U_T \mid \Fcal_t], \qquad t\le T<\infty.
\]
\end{proposition}

\begin{proof}
The proof is similar to that of Proposition~\ref{P:convex hull sticky}, but simpler.
\end{proof}

\section{Spaces of closed sets} \label{S:closed subsets}

Let $(\Xcal,d)$ be a complete separable metric space, and let ${\rm CL}(\Xcal)$ denote the collection of all nonempty closed subsets of $\Xcal$. In our applications, $\Xcal$ is either a countable set or $\R^d$, but we do not impose this yet. The distance between a point $x\in\Xcal$ and a subset $A\subseteq\Xcal$ is denoted by
\[
d(x,A)=\inf\{d(x,y)\colon y\in A\}.
\]
The {\em Wijsman topology} on ${\rm CL}(\Xcal)$ is the smallest topology for which the maps $A\mapsto d(x,A)$, $x\in\Xcal$, are all continuous; see~\cite{W:66}. It was proved by \citet[Theorem~4.3]{Beer:91} that with the Wijsman topology, ${\rm CL}(\Xcal)$ becomes a Polish space.

The space ${\rm CL}(\Xcal)$ is partially ordered by set inclusion. It is however not a lattice under union and intersection since it does not include the empty set. The space
\[
{\rm CL}_0(\Xcal) = {\rm CL}(\Xcal)\cup\{\emptyset\}
\]
on the other hand is a complete lattice with $\inf_\alpha A_\alpha=\bigcap_\alpha A_\alpha$ and $\sup_\alpha A_\alpha=\overline{\bigcup_\alpha A_\alpha}$ for arbitrary collections $\{A_\alpha\}\subseteq{\rm CL}_0(\Xcal)$. The Wijsman topology is extended to ${\rm CL}_0(\Xcal)$ by declaring a sequence of closed sets $A_n$ convergent to $\emptyset$ if $d(x,A_n)\to\infty$ for all $x\in\Xcal$. Equipped with the extended Wijsman topology, ${\rm CL}_0(\Xcal)$ is again a Polish space; see \citet[Theorem~4.4]{Beer:91}.

The spaces ${\rm CL}(\Xcal)$ and ${\rm CL}_0(\Xcal)$ are convenient from the point of view of stochastic analysis. The reason is a characterization due to~\citet{Hess:83,Hess:86} of the Borel $\sigma$-algebra on ${\rm CL}(\Xcal)$. Namely, the Borel $\sigma$-algebra coincides with the {\em Effros $\sigma$-algebra}, which is generated by the sets $\{A\in{\rm CL}(\Xcal)\colon A\cap V\ne\emptyset\}$, where $V$ ranges over the open subsets of $\Xcal$. This identification leads to the following lemma.

\begin{lemma} \label{L:range of SP}
Let $X=(X_t)_{t\ge0}$ be an $\Xcal$-valued c\`adl\`ag adapted process on a filtered measurable space $(\Omega,\Fcal,\F)$, whose filtration $\F$ is not necessarily right-continuous. Then the ${\rm CL}(\Xcal)$-valued process $Y=(Y_t)_{t\ge0}$ given by
\[
Y_t = \overline{\{X_s\colon s\le t\}}
\]
is adapted. The process is then also adapted when viewed as taking values in ${\rm CL}_0(\Xcal)$.
\end{lemma}

\begin{proof}
We need to argue that $\omega\mapsto Y_t(\omega)$ is $\Fcal_t$-measurable for each $t$. Using Hess's characterization, it suffices to inspect inverse images of sets $\{A\in{\rm CL}(\Xcal)\colon A\cap V\ne\emptyset\}$ with $V$ open. That is, we must check that the event
\[
F = \left\{\omega\in\Omega\colon \overline{\{X_s(\omega)\colon s\le t\}}\cap V \ne \emptyset\right\}
\]
lies in $\Fcal_t$. For a c\`adl\`ag process $X$, the set $\overline{\{X_s(\omega)\colon s\le t\}}\cap V$ is nonempty if and only if $X_s(\omega)\in V$ for some $s\le t$. Consequently,
\[
F = \{\omega\in\Omega\colon \tau(\omega) < t \text{ or } X_t(\omega)\in V\}, \qquad \text{where }\tau=\inf\{s\ge0\colon X_{s-} \in V\}.
\]
Since $X_-$ is left-continuous, $\tau$ is predictable, and hence $F\in\Fcal_t$; see \citet[Theorem~IV.73(b)]{DM:78}. The final assertion follows from the fact that $Y$ can never take the value $\emptyset$.
\end{proof}

The following result will be used later. Its proof illustrates the use of the two alternative descriptions of the Borel $\sigma$-algebra on ${\rm CL}(\Xcal)$. We use the notation
\begin{equation} \label{eq:A_epsilon}
A_\varepsilon=\{x\in \Xcal\colon d(x,A)\le\varepsilon\}
\end{equation}
for any $A\in{\rm CL}(\Xcal)$ and any $\varepsilon\ge0$. If $A=\emptyset$ then $A_\varepsilon=\emptyset$ by convention.

\begin{lemma} \label{L:gen prop}
\begin{enumerate}
\item\label{L:gen prop:mu} The map $A\mapsto\mu(A)$ from ${\rm CL}_0(\Xcal)$ to $\R$ is measurable, where $\mu$ is any finite measure on~$\Xcal$.
\item\label{L:gen prop:eps} The map $A\mapsto A_\varepsilon$ from ${\rm CL}_0(\Xcal)$ to itself is measurable for any $\varepsilon>0$.
\end{enumerate}
\end{lemma}

\begin{proof}
In both cases it suffices to show that the respective maps restricted to ${\rm CL}(\Xcal)$ are measurable.

\ref{L:gen prop:mu}: Using closedness of $A$ and the dominated convergence theorem, one obtains the equalities $\mu(A)=\int\bm 1_A(x)\mu(dx)=\int\bm 1_{\{0\}}(d(x,A))\mu(dx)=\lim_n \int(1-nd(x,A))^+\mu(dx)$, where $y^+$ denotes the positive part of~$y\in\R$. Each map $A\mapsto \int(1-nd(x,A))^+\mu(dx)$ is continuous, hence measurable, by definition of the Wijsman topology and the fact that $\delta\mapsto\int(1-n\delta)^+\mu(dx)$ is continuous due to the dominated convergence theorem. Thus the map $A\mapsto\mu(A)$ is the pointwise limit of real-valued measurable maps, and therefore itself measurable.

\ref{L:gen prop:eps}: One readily verifies $A_\varepsilon\cap V = A\cap V_\varepsilon$ for any open set $V$, where we define the open set $V_\varepsilon=\{x\in\Xcal: d(x,V)<\varepsilon\}$. Therefore $\{A\in{\rm CL}(\Xcal)\colon A_\varepsilon \cap V\} = \{A\in{\rm CL}(\Xcal)\colon A \cap V_\varepsilon\}$. The left-hand side is the inverse image of $\{A\colon A\cap V\ne\emptyset\}$ under the map $A\mapsto A_\varepsilon$, and the right-hand side lies in the Effros $\sigma$-algebra on ${\rm CL}(\Xcal)$. Measurability now follows from Hess's characterization.
\end{proof}

\subsection{Lattice operations}

In the following lemma, measurability is always understood with respect to the Borel $\sigma$-algebra. Since ${\rm CL}_0(\Xcal)$ is Polish, the Borel $\sigma$-algebra on ${\rm CL}_0(\Xcal)^k$ for $k\in\{2,3,\ldots,\infty\}$ coincides with the corresponding product $\sigma$-algebra.

\begin{lemma} \label{L:latop}
\begin{enumerate}
\item\label{L:latop:cup} The map $(A,B)\mapsto A\cup B$ from ${\rm CL}_0(\Xcal)^2$ to ${\rm CL}_0(\Xcal)$ is continuous.
\item\label{L:latop:closed} The set $\{(A,B)\colon A\subseteq B\}$ is closed in ${\rm CL}_0(\Xcal)^2$.
\item\label{L:latop:incr} If $A_n$ is a nondecreasing sequence in ${\rm CL}_0(\Xcal)$, meaning that $A_n\subseteq A_{n+1}$ for all~$n$, then $A_n$ converges to $\overline{\bigcup_n A_n}$ in ${\rm CL}_0(\Xcal)$.
\item\label{L:latop:cup2} The map $(A_n)\mapsto \overline{\bigcup_n A_n}$ from ${\rm CL}_0(\Xcal)^\infty$ to ${\rm CL}_0(\Xcal)$ is measurable.
\item\label{L:latop:cap} If $\Xcal$ is $\sigma$-compact, the map $(A_n)\mapsto \bigcap_n A_n$ from ${\rm CL}_0(\Xcal)^\infty$ to ${\rm CL}_0(\Xcal)$ is measurable.
\end{enumerate}
\end{lemma}

\begin{proof}
\ref{L:latop:cup}: Observe that $d(x,A\cup B) \le d(x,A)\wedge d(x,B)$ for all $x\in\Xcal$, where we use the convention $d(x,\emptyset)=\infty$. We claim that strict inequality is impossible. Indeed suppose $A\cup B\ne\emptyset$ and let $x_n\in A\cup B$ achieve $d(x,x_n)\to d(x,A\cup B)$. Suppose $A$ contains infinitely many of the $x_n$ (otherwise $B$ does, and we work with $B$ instead). Then $x_n\in A$ along a subsequence, so that $d(x,A)\le d(x,x_n)\to d(x,A\cup B)$. Therefore strict inequality is impossible, and we have $d(\fdot,A\cup B) = d(\fdot,A)\wedge d(\fdot,B)$. The stated continuity property now follows from the definition of the extended Wijsman topology.

\ref{L:latop:closed}: If $A_n\subseteq B_n$ and $(A_n,B_n)\to(A,B)$, then $B=\lim_n B_n=\lim_n A_n\cup B_n = A\cup B$ in view of~\ref{L:latop:cup}. Thus $A\subseteq B$, as required.

\ref{L:latop:incr}: The statement is obvious if $A_n=\emptyset$ for all $n$, so we suppose $A_n\ne\emptyset$ for some $n$, and then without loss of generality for all $n$. Define $A=\overline{\bigcup_n A_n}$ for ease of notation. Fix any $x\in\Xcal$. Since $A_n\subseteq A$, we have $d(x,A_n)\ge d(x,A)$ and hence $\lim_n d(x,A_n)\ge d(x,A)$. For the reverse inequality, pick any $\varepsilon>0$ and $y\in A$ such that $d(x,y)\le d(x,A)+\varepsilon$. Since $A$ is the closure of $\bigcup_nA_n$, there exists some $m$ and some $z\in A_m$ with $d(y,z)\le\varepsilon$. Consequently,
\[
d(x,A_m) \le d(x,z) \le d(x,y) + d(y,z) \le d(x,A) + 2\varepsilon.
\]
Since $d(x,A_n)$ is non-increasing, and since $\varepsilon$ was arbitrary, it follows that $\lim_n d(x,A_n)\le d(x,A)$. We deduce that $d(x,A_n)\to d(x,A)$ for all $x\in\Xcal$, which means that $A_n\to A$.

\ref{L:latop:cup2}: First note that the map $\varphi_k\colon{\rm CL}_0(\Xcal)^\infty\to{\rm CL}_0(\Xcal)$, $(A_n)\mapsto\bigcup_{n\le k}A_n$ is continuous, being a composition ${\rm CL}_0(\Xcal)^\infty \to {\rm CL}_0(\Xcal)^k \to {\rm CL}_0(\Xcal)$, $(A_n)\mapsto(A_1,\ldots,A_k)\mapsto \bigcup_{n\le k}A_n$ of two maps that are continuous by definition of the product topology and due to repeated use of~\ref{L:latop:cup}. By~\ref{L:latop:incr}, the map $(A_n)\mapsto \overline{\bigcup_n A_n}$ is the pointwise limit of the maps $\varphi_k$, and therefore measurable by \citet[Lemma~4.29]{AB:06}.

\ref{L:latop:cap}: Let $\varphi\colon(A_n)\mapsto \bigcap_n A_n$ denote the intersection map. We will prove that $\varphi^{-1}({\bf F})$ is a measurable subset of ${\rm CL}(\Xcal)^\infty$, hence of ${\rm CL}_0(\Xcal)^\infty$, for any measurable ${\bf F}\subseteq{\rm CL}(\Xcal)$. The same then holds for any measurable ${\bf F}\subseteq{\rm CL}_0(\Xcal)$, since $\varphi^{-1}(\{\emptyset\})=(\varphi^{-1}({\rm CL}(\Xcal)))^c$ is measurable. This readily implies the assertion.

We must thus argue that $\varphi^{-1}({\bf F})$ is measurable for any measurable ${\bf F}\subseteq{\rm CL}(\Xcal)$. In view of Hess's characterization of the Borel $\sigma$-algebra on ${\rm CL}(\Xcal)$ it suffices to consider sets of the form ${\bf F}=\{A\in{\rm CL}(\Xcal)\colon A\cap V\ne\emptyset\}$ with $V$ open. For such sets we have
\begin{equation} \label{eq:L:latop:phi-1}
\varphi^{-1}({\bf F}) = \big\{(A_n)\colon V\cap\bigcap_n A_n \ne\emptyset\big\} = \bigcup_m \big\{(A_n)\colon K_m\cap V\cap\bigcap_n A_n \ne\emptyset\big\},
\end{equation}
where $\{K_m\}_{m\in\N}$ is a compact cover of $\Xcal$, which exists by $\sigma$-compactness. Thus it suffices to prove measurability of any set of the form $\{(A_n)\colon K\cap V\cap\bigcap_n A_n \ne\emptyset\}$ with $V$ open and $K$ compact. Fix a countable dense subset $D\subseteq\Xcal$. We claim that for any $(A_n)\in{\rm CL}_0(\Xcal)^\infty$ we have
\begin{equation} \label{KcapV}
K\cap V\cap\bigcap_n A_n \ne\emptyset \qquad\Longleftrightarrow \qquad
\begin{minipage}[c][4.5em][c]{.40\textwidth}
\begin{center}
$\exists\varepsilon>0$ rational, $\forall k\in\N$, $\exists x_k\in D$, $d(x_k,K)\le k^{-1}$, $d(x_k,V^c)\ge\varepsilon$, and $\forall n\in\N$, $d(x_k,A_n)\le k^{-1}$.
\end{center}
\end{minipage}
\end{equation}
To prove ``$\Rightarrow$'', let $x\in K\cap V\cap\bigcap_n A_n$. Since $V$ is open, there exists some rational $\varepsilon>0$ such that $d(x,V^c)\ge 2\varepsilon$. Since $D$ is dense, there exist points $x_k\in D$ such that $d(x_k,x)\le k^{-1}\wedge\varepsilon$. The triangle inequality then yields $d(x_k,V^c)\ge d(x,V^c)-d(x_k,x)\ge\varepsilon$, and we have $d(x_k,K)\le d(x_k,x)\le k^{-1}$ as well as $d(x_k,A_n)\le d(x_k,x)\le k^{-1}$ for all~$n$. This proves the forward implication. To prove ``$\Leftarrow$'', let $\varepsilon>0$ and $x_k$, $k\in\N$, have the stated properties. Since $d(x_k,K)\le k^{-1}$, there exist $y_k\in K$ with $d(x_k,y_k)\le 2k^{-1}$. By compactness of $K$, we may pass to a subsequence and assume that $y_k\to x$ for some $x\in K$. Then also $x_k\to x$, and continuity of the distance function implies $d(x,V^c)\ge\varepsilon$ and $d(x,A_n)=0$ for all $n$. We conclude that $x\in K\cap V\cap\bigcap_n A_n$, which is therefore nonempty. This completes the proof of \eqref{KcapV}.

Now, observe that \eqref{KcapV} can be expressed as
\begin{equation} \label{eq:L:latop:cube}
\big\{(A_n)\colon K\cap V\cap\bigcap_n A_n \ne\emptyset\big\}
= \bigcup_{\substack{\varepsilon>0\\[.5ex]\varepsilon\in\Q}}  \bigcap_{k\in\N} \bigcup_{\substack{x_k\in D \text{ with}\\[1ex]d(x_k,K)\le k^{-1}\\[1ex]d(x_k,V^c)\ge\varepsilon}} \big\{(A_n)\colon d(x_k,A_n)\le k^{-1} \ \forall n\big\}.
\end{equation}
The right-hand side is formed through countable unions and intersections of sets of the form $\{(A_n)\colon d(x_k,A_n)\le k^{-1} \ \forall n\}$. Such a set is actually a cube ${\bf G}_k\times {\bf G}_k\times\cdots\subseteq{\rm CL}(\Xcal)^\infty$, where ${\bf G}_k=\{A\colon d(x_k,A)\le k^{-1}\}$ is the inverse image of $[0,k^{-1}]$ under the continuous map $A\mapsto d(x_k,A)$. We deduce that the right-hand side of \eqref{eq:L:latop:cube}, and hence the left-hand side, is measurable. Thus $\varphi^{-1}({\bf F})$ in \eqref{eq:L:latop:phi-1} is also measurable, as required.
\end{proof}

\begin{remark}
It appears unlikely to the author that $\sigma$-compactness is really be needed for measurability of the intersection map; dropping this assumption would be desirable and natural. However, it is interesting to note that there are some striking differences between unions and intersections. For instance, $A\cap B$ may be empty even if $A$ and $B$ are not. Also, the map $(A,B)\mapsto A\cap B$ is not continuous, even if one restrict to compact convex sets. Indeed, let $\Xcal=\R^2$, and let $A_n=\{(x_1,x_2):0\le x_1\le 1/n,\, x_2=nx_1\}$ be the straight line from the origin to the point $(1/n,1)$. Then $A_n\to B$, where $B=\{0\}\times[0,1]$ is the line from the origin to $(0,1)$. Thus $A_n\cap B=\{(0,0)\}$ does not converge to $(\lim_n A_n)\cap B=B$.
\end{remark}

\subsection{Vector space operations}

If $A$ and $B$ are subsets of a vector space, their sum is defined by $A+B=\{x+y\colon x\in A,\,y\in B\}$. This operation is associative and commutative, so the expression $A+B+C$ is unambiguous and equal to $A+C+B$, etc. Similarly, we define $\lambda A=\{\lambda x\colon x\in A\}$ for any scalar~$\lambda$. The dimension of an affine subspace $V$ is denoted $\dim(V)$, with the convention $\dim(\emptyset)=-1$.

\begin{lemma} \label{L:vec op}
Assume $\Xcal$ is a locally convex topological vector space.\footnote{Of course, the topology is assumed to coincide with the one generated by the given metric $d$.}
\begin{enumerate}
\item\label{L:vec op:+} The map $(A_1,\ldots,A_n)\mapsto \overline{A_1+\cdots+A_n}$ from ${\rm CL}_0(\Xcal)^n$ to ${\rm CL}_0(\Xcal)$ is measurable for any $n\in\N$.
\item\label{L:vec op:scal} The map $A\mapsto\lambda A$ from ${\rm CL}_0(\Xcal)$ to itself is measurable, where $\lambda$ is any scalar.
\item\label{L:vec op:conv} The map $A\mapsto\cconv(A)$ from ${\rm CL}_0(\Xcal)$ to itself is measurable.
\item\label{L:vec op:aff} The map $A\mapsto\caff(A)$ from ${\rm CL}_0(\Xcal)$ to itself is measurable.
\item\label{L:vec op:dim aff} The map $A\mapsto\dim(\aff(A))$ from ${\rm CL}_0(\Xcal)$ to $\R\cup\{\infty\}$ is lower semicontinuous.
\end{enumerate}
\end{lemma}

\begin{proof}
In each case, we only need to consider inverse images of measurable subsets of ${\rm CL}(\Xcal)$, since the inverse image of $\{\emptyset\}$ is obviously measurable for each of the given maps. The proofs all use Hess's characterization in terms of the Effros $\sigma$-algebra. Thus we inspect inverse images of the set $\{A\in{\rm CL}(\Xcal)\colon A\cap V\ne\emptyset\}$, where $V$ is any nonempty open subset of $\Xcal$.

\ref{L:vec op:+}: It suffices to consider the case $n=2$, as the general case follows by induction together with the fact that $\overline{A_1+A_2+A_3}=\overline{\overline{A_1+A_2}+A_3}$. Define the maps
\[
\varphi_\varepsilon\colon (A_1,A_2) \mapsto \overline{(A_1)_\varepsilon + (A_2)_\varepsilon}
\]
for any $\varepsilon\ge0$, where we use the notation~\eqref{eq:A_epsilon}. We may assume without loss of generality that the metric $d$ is translation invariant, see e.g.~\citet[Lemma~5.75]{AB:06}, in which case one readily verifies the inequalities
\[
d(x,A_1+A_2)-4\varepsilon \le d(x,(A_1)_\varepsilon+(A_2)_\varepsilon)\le d(x,A_1+A_2)
\]
for any $x\in\Xcal$ and $A_1,A_2\in {\rm CL}_0(\Xcal)$. It follows that $\lim_{\varepsilon\to0}\varphi_\varepsilon = \varphi_0$ pointwise with respect to the Wijsman topology. Thus it suffices to prove measurability of $\varphi_\varepsilon$ for $\varepsilon>0$. To this end, let $D\subseteq\Xcal$ be a countable dense subset. Observe that $\overline{(A_1)_\varepsilon+(A_2)_\varepsilon}$ intersects the open set $V$ if and only if $(A_1)_\varepsilon+(A_2)_\varepsilon$ does. Since each $(A_i)_\varepsilon$ has nonempty interior, this holds if and only if $x_1+x_2\in V$ for some points $x_i\in D\cap(A_i)_\varepsilon$. This can be expressed as follows:
\[
\{(A_1,A_2)\colon \overline{(A_1)_\varepsilon+(A_2)_\varepsilon}\cap V\ne\emptyset\} = \bigcup_{\substack{x_1,\,x_2\in D \\ x_1+x_2\in V}} \{(A_1,A_2)\colon d(x_i,A_i)\le\varepsilon,\, i=1,2\}.
\]
The right-hand side is a countable union of products of the sets $\{A_i\colon d(x_i,A_i)\le\varepsilon\}$, $i=1,2$, which are measurable since $d(x_i,\fdot)$ is continuous. Hence $\varphi_\varepsilon$ is measurable, as required.

\ref{L:vec op:scal}: If $\lambda=0$, the inverse image is either empty or all of ${\rm CL}(\Xcal)$, so we may suppose that $\lambda$ is nonzero. But then $\{A\colon (\lambda A)\cap V\ne\emptyset\}=\{A\colon A\cap (\lambda^{-1}V)\ne\emptyset\}$ is measurable since $\lambda^{-1}V$ is open whenever $V$ is.

\ref{L:vec op:conv}: Since $V$ is open, we have $V\cap\cconv(A)\ne\emptyset$ if and only if $V\cap\conv(A)\ne\emptyset$. This is equivalent to $\sum_i\lambda_i x_i \in V$ for some (finitely many) convex weights $\lambda_i$ and points $x_i\in A$. Again since $V$ is open, the $\lambda_i$ may be chosen rational. Therefore,
\[
\{A\colon \cconv(A)\cap V\ne\emptyset\} = \bigcup_n \bigcup_{\substack{\lambda_i\in\Q_+\text{ with}\\[0.5ex]\sum_{i=1}^n\lambda_i=1}} \{ A\colon (\lambda_1A + \cdots + \lambda_n A)\cap V\ne\emptyset\}.
\]
The right-hand side is measurable in view of \ref{L:vec op:+} and \ref{L:vec op:scal}, so the left-hand side is measurable as well.

\ref{L:vec op:aff}: The proof is identical to the one for the convex hull, except that the $\lambda_i$ are affine weights rather than convex weights, meaning that they sum to one but are not constrained to be nonnegative.

\ref{L:vec op:dim aff}: Choose any convergent sequence $A_n\to A$ and set $k=\dim(\aff(A))$. We need to show that $\liminf_n \dim(\aff(A_n)) \ge k$. For $k=-1$, i.e.~$A=\emptyset$, the statement is obvious. Suppose instead $0\le k<\infty$. Then there exist $k+1$ affinely independent points $x_0,\ldots,x_k\in A$. By definition of the extended Wijsman topology, $d(x_i,A_n)\to0$ for $i=0,\ldots,k$. Thus for all large $n$, $A_n$ also contains $k+1$ affinely independent points, whence $\dim(\aff(A_n))\ge k$. Finally, if $k=\infty$, the above argument replaced by an arbitrary $k'\in\N$ shows that $\dim(\aff(A_n))\ge k'$ for all large $n$, and thus $\lim_n \dim(\aff(A_n)) =\infty$.
\end{proof}

\subsection{The space of convex subsets of Euclidean space}

In this subsection we assume that $\Xcal=\R^d$ and that the metric comes from the norm, $d(x,y)=\|x-y\|$. We consider the subspace ${\rm CO}(\Xcal) \subset {\rm CL}_0(\Xcal)$ consisting of all closed convex subsets, equipped with the subspace topology and the associated Borel $\sigma$-algebra. The space ${\rm CO}(\Xcal)$ is again partially ordered by set inclusion, and is a complete lattice with $\inf_\alpha A_\alpha = \bigcap_\alpha A_\alpha$ and $\sup_\alpha A_\alpha = \cconv(\bigcup_\alpha A_\alpha)$ for arbitrary collections $\{A_\alpha\}\subseteq{\rm CO}(\Xcal)$. Note that ${\rm CO}(\Xcal)$ is a closed subset of ${\rm CL}_0(\Xcal)$. The following result shows that this complete lattice satisfies the assumptions imposed in Section~\ref{S:CI}.

\begin{theorem} \label{T:convex sets}
The complete lattice ${\rm CO}(\Xcal)$ satisfies assumptions \ref{A1}--\ref{A3}. A strictly increasing measurable map $\phi\colon{\rm CO}(\Xcal)\to\R$ is given by
\[
\phi(A) = \dim(\aff(A)) + \mu(A \mid \aff(A)),
\]
where $\mu(\fdot\mid V)$ is the distribution of an $\R^d$-valued standard Gaussian random variable conditioned to lie in the affine subspace $V$. We set $\mu(\emptyset\mid\emptyset)=0$ by convention.
\end{theorem}

\begin{proof}
Due to Lemma~\ref{L:latop}\ref{L:latop:closed} the set $\{(A,B)\in{\rm CO}(\Xcal)^2\colon A\subseteq B\}$ is closed in ${\rm CO}(\Xcal)^2$ and hence measurable. Thus Assumption~\ref{A1} holds. Lemma~\ref{L:latop}\ref{L:latop:cap} yields that the countable infimum map is measurable, and Lemma~\ref{L:latop}\ref{L:latop:cup2} together with Lemma~\ref{L:vec op}\ref{L:vec op:conv} yield that the countable supremum map is measurable. Thus Assumption~\ref{A2} holds. Next, we claim that the map $\phi$ is strictly increasing. To see this, first note that $\phi(\emptyset)=0$ and $\phi(A)\ge1$ if $A\ne\emptyset$. Next, let $A\subsetneq B$ be two nonempty convex sets. If $\dim(\aff(A))<\dim(\aff(B))$ then $\phi(A) \le \dim(\aff(B)) - 1 + \mu(A\mid\aff(A)) \le \dim(\aff(B)) < \phi(B)$. On the other hand, if $\dim(\aff(A))=\dim(\aff(B))$, then the two affine hulls coincide and we denote them both by~$V$. Since $A$ is strictly contained in $B$ and both sets are convex and closed, $B\setminus A$ contains a set which is open in $V$. Therefore $\phi(B)-\phi(A)=\mu(B\setminus A\mid V)>0$. Finally, to see that $\phi$ is measurable, first note that $A\mapsto\dim(\aff(A))$ is measurable since it is lower semicontinuous by Lemma~\ref{L:vec op}\ref{L:vec op:dim aff}. Next, observe that
\[
\mu(A\mid \aff(A)) = \lim_{\varepsilon\downarrow0} \frac{\mu(A_\varepsilon)}{\mu(\aff(A)_\varepsilon)}, \qquad A\ne\emptyset,
\]
where $\mu(\fdot)$ is the standard Gaussian distribution on $\R^d$. Therefore, by Lemma~\ref{L:gen prop}\ref{L:gen prop:mu}--\ref{L:gen prop:eps} and Lemma~\ref{L:vec op}\ref{L:vec op:aff}, the map $A\mapsto\mu(A\mid\aff(A))$ is a limit of measurable maps, and hence itself measurable.
\end{proof}

\begin{lemma} \label{L:conv is adapted}
Let $X=(X_t)_{t\ge0}$ be an $\R^d$-valued c\`adl\`ag adapted process on a filtered measurable space $(\Omega,\Fcal,\F)$, whose filtration $\F$ is not necessarily right-continuous. Then the ${\rm CO}(\R^d)$-valued process $Y=(Y_t)_{t\ge0}$ given by
\[
Y_t = \cconv(X_s\colon s\le t)
\]
is adapted.
\end{lemma}

\begin{proof}
This follows from Lemma~\ref{L:range of SP} and Lemma~\ref{L:vec op}\ref{L:vec op:conv}.
\end{proof}

\subsection{The space of subsets of a countable set}

In this subsection we assume that $\Xcal$ is countable set equipped with the discrete metric $d(x,y)=\bm 1_{\{y\}}(x)$. Then every subset of $\Xcal$ is closed, so $2^\Xcal={\rm CL}_0(\Xcal)$. This space is partially ordered by set inclusion, and is a complete lattice under union and intersection. Furthermore, it satisfies the assumptions of Section~\ref{S:CI}.

\begin{theorem} \label{T:countable set}
The complete lattice $2^\Xcal$ satisfies assumptions \ref{A1}--\ref{A3}. A strictly increasing measurable map $\phi:2^\Xcal\to\R$ is given by
\[
\phi(A) = \sum_{x\in A}w(x),
\]
where $\{w(x):x\in\Xcal\}$ is a countable set of strictly positive numbers summing to one.
\end{theorem}

\begin{proof}
Assumptions~\ref{A1} and~\ref{A2} follows directly from Lemma~\ref{L:latop}\ref{L:latop:closed}, \ref{L:latop:cup2}, and~\ref{L:latop:cap}. The map $\phi$ is clearly strictly increasing. To see that it is measurable, write $\phi(A)=\sum_{x\in\Xcal}\bm 1_A(x)w(x)=\sum_{x\in\Xcal}\bm (1-d(x,A))w(x)$ and observe that $A\mapsto d(x,A)$ is continuous and hence measurable.
\end{proof}

\appendix

\section{Extension of Dedekind complete lattices}

\begin{proposition} \label{P:ext Dedekind}
Let $(S_0,\le)$ be a Dedekind complete lattice equipped with a $\sigma$-algebra~$\Scal_0$, and assume that the following conditions hold:
\begin{enumerate}[label={\rm(A\arabic*$_0$)}]
\item\label{A1_0} The set $\{(x,y)\in S_0^2\colon x\le y\}$ lies in the product $\sigma$-algebra $\Scal_0^2$.
\item\label{A2_0} For every measurable subset $A\in\Scal_0$, the sets
\[
\{(x_1,x_2,\ldots)\colon \sup\{x_1,x_2,\ldots\} \in A\} \quad\text{and}\quad \{(x_1,x_2,\ldots)\colon \inf\{x_1,x_2,\ldots\} \in A\}
\]
both lie in $\Scal_0^\infty$.\footnote{Here $\sup$ and $\inf$ refer to the operations on $S_0$. If $(x_1,x_2,\ldots)$ is a sequence in $S_0$ for which $\sup_n x_n$ does not exist, then the condition $\sup_n x_n \in A$ is by convention not satisfied. Similarly for $\inf_n x_n$. In particular, by considering $A=S_0$, condition \ref{A2_0} implies that the set of sequences which admit a supremum and/or infimum is measurable, i.e.~lies in $\Scal_0^\infty$.}
\item\label{A3_0} There exists a strictly increasing measurable map $\phi_0:S_0\to\R$.
\end{enumerate}
Define $S=S_0\cup\{-\infty,+\infty\}$, where $-\infty$ and $+\infty$ are not elements of $S_0$, and define $\Scal=\Scal_0\vee\sigma(\{-\infty\},\{+\infty\})$. Extend the order $\le$ to $S$ by declaring $-\infty$ ($+\infty$) a lower (upper) bound on $S_0$. Then $(S,\le)$ with the $\sigma$-algebra $\Scal$ satisfies \ref{A1}--\ref{A3}.
\end{proposition}

\begin{proof}
The set $\{(x,y)\in S^2\colon x\le y\}$ is the union of $\{(x,y)\in S_0^2\colon x\le y\}$, $\{-\infty\}\times S$, and $S\times\{+\infty\}$. It is therefore measurable, so~\ref{A1} holds. Next, let $A_{\rm sup}\subseteq S_0^\infty$ be the set of all sequences of elements in $S_0$ which admit a supremum in $S_0$. Condition~\ref{A2_0} implies that this set is measurable, $A_{\rm sup}\in\Scal_0^\infty$. It is easy to check that the countable supremum map $\varphi$ on $S^\infty$ is given by
\[
\varphi((x_n)_{n\in\N}) = \begin{cases} \sup_n x_n, & (x_n)_{n\in\N}\in A_{\rm sup}\\ +\infty, & (x_n)_{n\in\N}\notin A_{\rm sup}. \end{cases}
\]
It then follows from \ref{A2_0} that $\varphi$ is $(\Scal^\infty,\Scal)$-measurable. The countable infimum map on $S^\infty$ is similarly shown to be measurable. This proves \ref{A2_0}. Finally, by replacing $\phi_0$ by $\frac{2}{\pi}\arctan(\phi_0)$ if necessary, we may assume that $\phi_0$ takes values in the interval $[-1,1]$. The map $\phi:S\to\R$ defined by $\phi(x)=\phi_0(x)$ for $x\in S_0$, $\phi(-\infty)=-2$, and $\phi(+\infty)=+2$, is then a strictly increasing measurable map. Thus~\ref{A3} holds.
\end{proof}

\bibliographystyle{plainnat}
\bibliography{bibl}

\end{document}